\documentclass[12pt]{article}
\usepackage{amsmath, amscd, amssymb, latexsym, epsfig, color, amsthm}

\newtheorem{theorem}{Theorem}[section]
\newtheorem{proposition}[theorem]{Proposition}
\newtheorem{question}[theorem]{Question}
\newtheorem{corollary}[theorem]{Corollary}
\newtheorem{conjecture}[theorem]{Conjecture}
\newtheorem{lemma}[theorem]{Lemma}

\theoremstyle{definition}
\newtheorem{defn}[theorem]{Definition}

\newtheorem{remark}[theorem]{Remark}
\newenvironment{proofof}[1]
	{\begin{trivlist}\item {\it Proof of {#1}.}}{\qed\end{trivlist}}
\newenvironment{proofof*}[1]
	{\begin{trivlist}\item {\it Proof of {#1}.}}{\end{trivlist}}
\newenvironment{case}[1]{\begin{trivlist}\item {\bf #1}}{\end{trivlist}}

\def\N{\mathbb{N}}

\def\Z{\mathbb{Z}}
\def\m{\mathfrak{m}}

\def\V{\mathbf{V}}
\def\W{\mathbf{W}}

\def\M{\mathcal{M}}
\def\S{\mathcal{S}}

\newcommand\lk{\mbox{\upshape lk}\,}
\newcommand\inc{\iota}

\newcommand\local[2]{_{#2}(#1)}
\newcommand\kk{\Bbbk}
\newcommand\mm{\mathfrak{m}}
\newcommand\pp{\mathfrak{p}}
\newcommand\xx{\mathbf{x}}
\newcommand\EE{\mathcal{E}}
\newcommand\FF{\mathcal{F}}
\newcommand\NN{\mathbb{N}}
\newcommand\OO{\mathcal{O}}
\newcommand\PP{\mathbb{P}}
\newcommand\ZZ{\mathbb{Z}}

\newcommand\spot{{\raisebox{.25ex}{\tiny$\scriptscriptstyle\bullet$}}}
\newcommand\nothing{\varnothing}
\renewcommand\iff{\Leftrightarrow}

\DeclareMathOperator\gr{\mathrm{gr}}

\DeclareMathOperator\ext{\mathrm{Ext}}
\DeclareMathOperator\tor{\mathrm{Tor}}
\DeclareMathOperator\Hom{\mathrm{Hom}}
\DeclareMathOperator\Ext{\mathcal{E}\!\mathit{xt}}

\DeclareMathOperator\cost{\mathrm{cost}}
\DeclareMathOperator\coker{\mathrm{coker}}

\begin{document}

\title{\vspace{-1ex}Face rings of simplicial complexes with singularities}
\author{
Ezra Miller
\thanks{Research partially supported by NSF CAREER grant DMS-0449102}\\
\small Mathematics Department, Box 90320\\[-.8ex]
\small Duke University, Durham, NC 27707, USA\\[-.8ex]
\small \texttt{ezra@math.duke.edu}
\and Isabella Novik
\thanks{Research partially supported by Alfred P.~Sloan Research
Fellowship and NSF grant DMS-0801152}\\
\small Department of Mathematics, Box 354350\\[-0.8ex]
\small University of Washington, Seattle, WA 98195-4350, USA\\[-0.8ex]
\small \texttt{novik@math.washington.edu}
\and Ed Swartz
\thanks{Research partially supported by NSF grant DMS-0900912}\\
\small Department of Mathematics, \\[-0.8ex]
\small Cornell University, Ithaca NY, 14853-4201, USA\\[-0.8ex]
\small \texttt{ebs22@cornell.edu}
\date{15 January 2010}
}

\maketitle

\begin{abstract}
The face ring of a simplicial complex modulo $m$ generic linear forms
is shown to have finite local cohomology if and only if the link of
every face of dimension~$m$ or more is \emph{nonsingular},
i.e., has the homology of a wedge of spheres of the expected
dimension.  This is derived from an enumerative result for local
cohomology of face rings modulo generic linear forms, as compared with
local cohomology of the face ring itself. The enumerative result is
generalized in slightly weaker form to squarefree modules.  A concept
of \emph{Cohen--Macaulay in codimension $c$} is defined and
characterized for arbitrary finitely generated modules and coherent
sheaves. For the face ring of an $r$-dimensional complex $\Delta$, it
is equivalent to nonsingularity of $\Delta$ in dimension $r-c$; for a
coherent sheaf on projective space, this condition is shown to be
equivalent to the same condition on any single generic hyperplane
section.  The characterization of nonsingularity in dimension~$m$ via
finite local cohomology thus generalizes from face rings to arbitrary
graded modules.
\end{abstract}

\noindent{\it 2010 Mathematics Subject Classification.\hspace{-1.2pt}}
13F55,~05E40,~05E45,~13H10,~14M05,~13D45,~13C14


\section{Introduction}

In the 1970's Reisner (building on unpublished work of Hochster) and
Stanley revolutionized the study of face enumeration of simplicial
complexes through their use of the \emph{face ring}, also called the
\emph{Stanley--Reisner ring}.  Reisner proved that the face ring of a
simplicial complex is Cohen--Macaulay if and only if the link of every
face in the simplicial complex, including the empty face, is
\emph{nonsingular}, by which we mean that all reduced cohomology
groups, except possibly in the maximum dimension,
vanish~\cite{Reisner}.  Stanley used this to completely characterize
the $f$-vectors of Cohen--Macaulay complexes~\cite{St77}.

A natural question which follows these results is, ``What happens if
singularities are allowed?''  The weakest relaxation possible is to
permit nontrivial cohomology of the whole complex (the link of the
empty face) in lower dimensions.  Schenzel proved that for pure
complexes, the face ring is Buchsbaum if and only if this is the only
additional cohomology of links \cite{Sch}.  The primary tool used in
the proof of both Reisner's and Schenzel's theorems is local
cohomology of the face ring with respect to the irrelevant ideal.  In
the Cohen--Macaulay case  all but the top local cohomology modules vanish, 
while in the Buchsbaum case these modules have finite
dimension as vector spaces.  Rings with this property, that is, those
whose local cohomology modules below their Krull dimension have finite
dimension, are called \emph{generalized Cohen--Macaulay rings} or
\emph{rings with finite local cohomology}.

Our first goal is to extend these ideas to arbitrary singularities.
The main result along simplicial lines, Theorem~\ref{main}, says that
if the dimension of the singular set is $m-1$, then except for top
cohomological degree, all of the local cohomology modules of the face
ring modulo $m$ generic linear forms have finite dimension and vanish
outside of $\Z$-graded degrees $i$ for $0 \le i \le m$.
Theorem~\ref{main} is verified easily after Theorem~\ref{isomorphism}.
Theorem~\ref{isomorphism} is
an enumerative result for Hilbert series of local cohomology of face
rings modulo generic linear forms proved using a combinatorial description
of the finely graded module structure of the local cohomology $H^i_\m
(\kk[\Delta])$ due to Gr\"abe \cite{Grabe}, which we review in
Section~\ref{s:grabe}.  We treat isolated singularities
in~Section~\ref{isolated sings}, and the general case in
Section~\ref{proof}.

Theorem~\ref{main} is not limited to a simplicial phenomenon; it is an
instance of a rather general phenomenon in commutative algebra and
algebraic geometry, stated and proved here in the language of sheaves
on projective schemes as Theorem~\ref{t:codim}, and reinterpreted in
terms of commutative algebra in Theorem~\ref{t:algebraic}.  The first
observation, made precise in Theorem~\ref{t:singdim}, is that
singularity in a fixed dimension is equivalent to a condition that we
call \emph{Cohen--Macaulay in a fixed codimension}: roughly speaking,
the local rings at all points of some fixed codimension are
Cohen--Macaulay.  At first sight, this condition sounds like Serre's
condition~$S_k$, which has been treated in combinatorial settings such
as affine semigroup rings and toric varieties~\cite{SchSch}, but it is
subtly different; see Remark~\ref{r:module}.  The second, harder and
deeper homological observation is that the Cohen--Macaulay condition
in a fixed codimension satisfies a Bertini-type persistence under
generic hyperplane section, but also the converse: if a single generic
hyperplane section of a coherent sheaf is Cohen--Macaulay in
codimension~$c$, then so is the original sheaf; this is made precise
in Theorem~\ref{t:codim}.

Having seen in Section~\ref{modules} that Theorem~\ref{main} extends
beyond the simplicial realm, we do the same for the enumerative result
from which it follows: in Theorem~\ref{t:sqfree}, we extend
Theorem~\ref{isomorphism} to the squarefree modules introduced by
Yanagawa~\cite{yanagawa}.

\section{Singularity dimension and finite local cohomology}

For all undefined terminology we refer our readers to \cite{BrHe, cca,
St96}.  Throughout, $\Delta$ is a simplicial complex of dimension $d -
1$ with vertex set $[n]=\{1, \dots, n\}$, and $\kk$ is a field.  If $F
\in \Delta$ is a face, then the {\it link} of $F$ is
$$%
  \lk F = \{G \in \Delta: F \cap G = \nothing, F \cup G \in \Delta\}.
$$
In particular, $\lk \nothing = \Delta$.

\begin{defn}
The face $F$ is {\it nonsingular} with respect to~$\kk$ if the reduced
cohomology $\tilde H^i(\lk F; \kk)$ of the link of~$F$ with
coefficients in~$\kk$ vanishes for all $i < d-1-|F|$.  Otherwise $F$
is a {\it singular} face.  The {\it singularity dimension} of $\Delta$
is the maximum dimension of a singular face.  If $\Delta$ has no
singular faces, then $\Delta$ is {\it Cohen--Macaulay} (over $\kk$)
and we (arbitrarily) declare the singularity dimension of the complex
to be $-\infty$.
\end{defn}

For a field $\kk$, which we assume is infinite but of arbitrary
characteristic, the {\it face ring} of~$\Delta$ (also known as the
{\it Stanley--Reisner} ring) is
$$%
  \kk[\Delta] = \kk [x_1, \dots, x_n]/I_\Delta,
$$
where the \emph{Stanley--Reisner ideal} is
$$%
   I_\Delta = ( x_{i_1}\cdots x_{i_k}: \{i_1, \dots, i_k\} \notin \Delta).
$$

For any module $M$ over the polynomial ring $S = \kk[x_1, \ldots,
x_n]$, we use $H^i_\m(M)$ to denote the $i$-th local cohomology module
of $M$ with respect to the irrelevant ideal $\m = (x_1, \ldots, x_n)$.

\begin{defn}
Let $M$ be an $S$-module of Krull dimension~$d$.  Then $M$ is a {\it
module with finite local cohomology} (or a {\it generalized
Cohen--Macaulay module}) if $H^i_\m(M)$ has finite dimension as a
$\kk$-vector space for all $i < d$.
\end{defn}

Modules with finite local cohomology were introduced in \cite{CST},
\cite{StVo}, and \cite{Trung}.  Connections between face rings and
modules with finite local cohomology have been studied in
\cite{GotoYuk} and~\cite{Yuk}.

Here we show that the connection between algebraic properties of
$\kk[\Delta]$ and the singularities of the complex is given by the
following.

\begin{theorem}\label{main}%
The singularity dimension of a simplicial complex~$\Delta$ is less
than~$m$ if and only if for all sets $\{\theta_1, \dots, \theta_m\}$
of $m$ generic linear forms the quotient $\kk[\Delta]/(\theta_1,
\dots, \theta_m)$ is a ring with finite local cohomology.
\end{theorem}

We give two proofs: the one in Section~\ref{proof} uses simplicial
techniques on local cohomology; the one in Section~\ref{modules}
frames it as a special case of the general theory of modules and
sheaves that are Cohen--Macaulay in a fixed codimension, via
Theorem~\ref{t:singdim}.

\section{Local cohomology of face rings}\label{s:grabe}

This section explains Gr\"abe's~results on the combinatorial structure
of the local cohomology of a face ring as a $\ZZ^n$-graded module over
the polynomial ring $S = \kk[x_1,\ldots,x_n]$.

Denote by $|\Delta|$ the geometric realization of $\Delta$.  For a
face $\tau\in \Delta$, the subcomplex
$$%
  \cost \tau = \{\sigma \in \Delta \, : \sigma\not\supset \tau\}
$$
is the \emph{contrastar} of~$\tau$.  The $i$-th \emph{local
cohomology} of $\Delta$ along~$\tau$ is the simplicial $i$-th
cohomology
$$
  H^i\local{\Delta}{\tau} := H^i(\Delta, \cost \tau)
$$
of the pair $(\Delta, \cost\tau)$ with coefficients in~$\kk$.  For
$\tau \subseteq \sigma \in \Delta$, the inclusion $\inc: \cost\tau \to
\cost\sigma$ induces a contravariant map $\inc^*:
H^i\local{\Delta}{\sigma} \to H^i\local{\Delta}{\tau}$.

For additional notation, an integer vector $U=(u_1, \ldots, u_n)\in
\Z^n$ has \emph{support}
$$%
  s(U) = \{\ell : u_\ell\neq 0\}\subseteq [n]
$$
and sum $|U| = \sum_{\ell=1}^n u_\ell$.  Fix the standard basis
$\{e_\ell\}_{\ell=1}^n$ for $\Z^n$, and write $\N$ for the set of
nonnegative integers.

We consider the $\Z^n$-grading of $\kk[x_1,\ldots, x_n]$ obtained by
declaring the variable $x_\ell$ to have degree~$e_\ell$.  This grading
refines the usual $\Z$-grading and induces a $\Z^n$-grading on
$\kk[\Delta]$ and its local cohomology modules; see
\cite[Chapter~13]{cca}, for example.  Thus, in particular,
$H^i_\m(\kk[\Delta])_j = \bigoplus_{|U|=j} H^i_\m(\kk[\Delta])_U$, and
multiplication by $x_\ell$ is a $\kk$-linear map from
$H^i_\m(\kk[\Delta])_U$ to $H^i_\m(\kk[\Delta])_{U+e_\ell}$ for all
$U\in \Z^n$.

\begin{theorem}[Gr\"abe] \label{Grabe}
The following is an isomorphism of $\Z^n$-graded  $\kk[\Delta]$-modules:
\begin{equation} \label{Hochster}
  H^{i}_\m(\kk[\Delta])\cong
  \bigoplus_{\substack{-U\in \N^n\\s(U)\in\Delta}} \M^{i}_U,
  \quad \text{where} \quad \M^{i}_U=H^{i-1}\local{\Delta}{s(U)},
\end{equation}
and the $\kk[\Delta]$-structure on the $U$-th component of 
the right-hand side is given by 
$$%
\cdot x_\ell = \left\{ \begin{array}{lll}
\text{$0$-map} & \text{if } \ell\notin s(U)\\
\text{identity map} & \text{if } \ell\in s(U) \text{ and } \ell\in s(U+e_\ell)\\
\inc^*: H^{i-1}\local{\Delta}{s(U)} \to H^{i-1}\local{\Delta}{s(U+e_\ell)} &
\text{otherwise}.
\end{array}
\right.
$$
\end{theorem}
We note that the isomorphism of (\ref{Hochster}) on the level of
vector spaces (rather than $\kk[\Delta]$-modules) is attributed to
Hochster in \cite[Section II.4]{St96}, and that
$$%
  H^i_\m(\kk[\Delta])_0\cong \M^i_{(0,\ldots,0)} =
  H^{i-1}\local{\Delta}{\nothing} = \tilde{H}^{i-1}(\Delta;\kk)
$$
is the reduced cohomology of~$\Delta$ itself, by definition.

\begin{corollary}\label{comb-char}
Let $\Delta$ be a simplicial complex of dimension $d - 1$.  The
singularity dimension of~$\Delta$ is less than~$m$ if and only if the
standard $\ZZ$-graded Hilbert series $F(H^i_\mm(\kk[\Delta]),
1/\lambda)$ of $H^i_\mm(\kk[\Delta])$ has a pole at $\lambda = 1$ of
order at most~$m$ for $i < d$.
\end{corollary}
\begin{proof}
An easy computation using eq.~\eqref{Hochster} implies that for $i<d$,
$$%
  F(H^i_\mm(\kk[\Delta]), 1/\lambda)
= \sum_{F\in\Delta} \dim_{\kk} H^{i-1}\local{\Delta}{F} \cdot
  \frac{\lambda^{|F|}}{(1-\lambda)^{|F|}}
= \sum_{F\in\Delta} \dim_{\kk} \tilde{H}^{i-|F|-1}(\lk F; \kk) \cdot
  \frac{\lambda^{|F|}}{(1-\lambda)^{|F|}}
$$
(cf.~\cite[Theorem II.4.1]{St96}).  The result now follows from the
definition of singularity dimension. 
\end{proof}

\section{Simplicial isolated singularities} \label{isolated sings}

Before proceeding to the proof of Theorem~\ref{main}, we consider a
special case.

\begin{defn}
$\Delta$ has {\it isolated singularities} if the singularity dimension
of $\Delta$ is zero.
\end{defn}

For the rest of this section we assume that $\Delta$ has isolated
singularities.  To begin with, we compute $H^i_\m(\kk[\Delta])$ for $i
< d$.  Since $\Delta$ has isolated singularities, Theorem~\ref{Grabe}
says that
$$%
  H^i_\m(\kk[\Delta])_U = 
  \begin{cases}
    \tilde{H}^{i-1}(\Delta; \kk) & \text{if } U=0 \\
    H^{i-1}\local{\Delta}{\{t\}} & \text{if } s(U) = \{t\}, \, U \in -\N^n\\
    0                            & \text{otherwise.}
  \end{cases}  
$$

Let $\theta =\sum^n_{t=1} a_t x_t$ be a linear form in $S$ with $a_t
\neq 0$ for all $t$.  We use the exact sequence
$$%
  \quad 0 \to \kk[\Delta] \stackrel{\cdot \theta}{\to}
  \kk[\Delta] \to \kk[\Delta]/(\theta) \to 0
$$
to compute $H^i_\m(\kk[\Delta]/(\theta))$ for $i < d-1$.  (The Krull
dimension of $\kk[\Delta]/(\theta)$ is $d - 1$.)  Since the map $\cdot
\theta$, that is multiplication by~$\theta$, is injective for any face
ring, the above sequence is in fact short exact.  The corresponding
long exact sequence in local cohomology is
\begin{equation*}
  \dots \to H^i_\m(\kk[\Delta])(-1) \stackrel{(\cdot \theta)^\ast}{\to} 
  H^i_\m(\kk[\Delta]) \to 
  H^i_\m(\kk[\Delta]/(\theta)) \stackrel{\delta}{\to} 
  H^{i+1}_\m(\kk[\Delta])(-1) \stackrel{(\cdot \theta)^\ast}{\to} 
  \dots \, .
\end{equation*}
Here, $\delta$ is the connecting homomorphism,
the notation $(-1)$ indicates a $\ZZ$-graded shift up by~$1$ (thus
$M(-1)_j = M_{j-1}$ for any $\ZZ$-graded $S$-module~$M$), and $(\cdot
\theta)^\ast$ is the map induced by multiplication, hence it is just the
module action of multiplication by~$\theta$ on $H^i_\m(\kk[\Delta])$.

Let $f^i: \bigoplus^n_{t=1} H^{i}\local{\Delta}{\{t\}} \to
H^i\local{\Delta}{\nothing}$ be the map on simplicial local cohomology
defined~by
$$%
  f^{i}= \sum_{t=1}^n a_t \cdot \inc^\ast\left[H^{i}\local{\Delta}{\{t\}}
  \to  H^i\local{\Delta}{\nothing}\right].
$$
By Gr\"abe's description of the $S$-module structure of the local
cohomology modules, using the standard $\ZZ$-grading (in which $\deg
x_t = 1 \in \ZZ$ for all~$t$) and the above long exact sequence,
\begin{equation} \label{isolated}
H^i_\m(\kk[\Delta]/(\theta))_j \cong
  \begin{cases}
  0                                   & \text{if } j<0 \\
  (\coker f^{i-1}) \oplus (\ker f^i)  & \text{if } j=0 \\
  H^i\local{\Delta}{\nothing}         & \text{if } j=1
  \end{cases}  \qquad (\forall i<d-1).
\end{equation}
In particular, $H^i_\m(\kk[\Delta]/(\theta))$ has finite dimension as
a vector space over~$\kk$ for $i<d-1$ .
 
Under certain conditions it is possible to extract even more precise
information.  We say that $\Delta$ has {\it homologically isolated
singularities} if in dimensions $i$ in the range $0 \le i \le d-2$,
the kernel of the above map $f^i$ decomposes as a direct sum:
$$%
  \ker f^i = \bigoplus^n_{t=1} \ker\inc^\ast\left[H^{i}\local{\Delta}{\{t\}}
  \to  H^i\local{\Delta}{\nothing}\right].
$$
Equivalently, $\Delta$ has homologically isolated singularities if for
$0 \le i \le d-2$, the images
$\inc^\ast\big(H^{i}\local{\Delta}{\{t\}}\big)$ for $t = 1,\ldots,n$
are linearly independent subspaces of $H^i\local{\Delta}{\nothing}$.
Evidently, any complex with only one isolated singularity has
homologically isolated singularities.  Other possibilities include
(among many) pinching off homologically independent handles of a
handlebody or coning off boundary components of a
manifold-with-boundary having appropriate homological properties.

Suppose $\Delta$ has homologically isolated singularities.  In
addition to the vector space structure of (\ref{isolated}), the
$\ZZ$-grading implies that the $S$-module structure of
$H^i_\m(\kk[\Delta]/(\theta))_j$ is trivial for $j=1$ and for $j=0$
restricted to $\coker f^{i-1}$.  Furthermore, on $\ker f^i \subseteq
H^i_\m(\kk[\Delta]/(\theta))_0$ the $S$-module structure is still
induced by $\inc$, and hence also trivial.  Therefore, by
\cite[Proposition~I.3.10]{StVo}, $\kk[\Delta]/(\theta)$ is Buchsbaum.
Thus, the Hilbert function of $\kk[\Delta]/(\theta_1, \dots,
\theta_d)$ for any linear system of parameters (l.s.o.p.)\
$\{\theta_1, \dots, \theta_d\}$ is completely determined by the
topology of $\Delta$ and its $f$-vector.  (Indeed, since $\kk$ is
infinite, any l.s.o.p.~for $\Delta$ will contain in its $\kk$-linear
span an element which is nonzero on every vertex and this one can be
used as $\theta_1$.)  This is a well known property of
Cohen--Macaulay~\cite{Reisner} and, more generally, Buchsbaum
complexes~\cite{Sch}.  For an arbitrary complex $\Delta$, distinct
linear systems of parameters can produce different Hilbert functions
for the Artinian quotient $\kk[\Delta]/(\theta_1, \dots, \theta_d)$.
However, for sufficiently generic linear systems of parameters, the
Hilbert function of $\kk[\Delta]/(\theta_1, \dots, \theta_d)$ is
constant.  A natural question is the following.

\begin{question}
Let $\Delta_1$ and $\Delta_2$ be two simplicial complexes whose
geometric realizations are homeomorphic and whose $f$-vectors are
identical.  Are the Hilbert functions of their Artinian quotients
obtained from sufficiently generic linear systems of parameters the
same?
\end{question}
 
\section{Main enumerative theorem for face rings}\label{proof}

We now proceed to the proof of the main theorem.  As before, let
$\Delta$ be a $(d-1)$-dimensional simplicial complex on $[n]$.  Let
$\theta_1,\ldots, \theta_d$ be $d$ generic linear forms in $S$ with
$\theta_p=\sum_{t=1}^n a_{t,p}x_t$.  In particular, we assume that
every square submatrix of the $n\times d$-matrix $A=(a_{t,p})$ is
nonsingular and the $\theta$'s satisfy the prime avoidance argument
in the proof of Theorem~\ref{isomorphism}~below.
 
Each $\theta_p$ acts on $\kk[\Delta]$ by multiplication,
$\cdot\theta_p: \kk[\Delta] \to \kk[\Delta]$.  This map, in turn,
induces the map $(\cdot \theta_p)^{\ast}=\cdot \theta_p :
H^\ell_{\m}(\kk[\Delta]) \to H^\ell_{\m}(\kk[\Delta])$ that increases
the $\ZZ$-grading by~$1$.  The key objects in the proof are the
kernels of these maps and their intersections:
$$%
\ker^\ell_{m, i} := \bigcap_{p=1}^m \left(\ker (\cdot \theta_p)^{\ast} \, : 
\, H^\ell_\m(\kk[\Delta])_{-(i+1)}
\to  H^\ell_\m(\kk[\Delta])_{-i}\right) \quad \mbox{and} \quad 
\ker^\ell_{m}:=\bigoplus_{i\in\ZZ} \ker^\ell_{m, i}.
$$
Thus $\ker^\ell_m$ is a graded submodule of $H^\ell_\m(\kk[\Delta])$
and $\ker^\ell_{m,i}$ is simply $(\ker^\ell_m)_{-(i+1)}$.

What are the dimensions of these kernels?  If $m=0$, then
$\ker^\ell_{0, i} = H^\ell_\m(\kk[\Delta])_{-(i+1)}$, and
Theorem~\ref{Grabe} implies that for $i\geq 0$,
\begin{eqnarray}  \nonumber
\dim_{\kk} \ker^\ell_{0, i}
& = & \sum_{F\in\Delta} |\{U\in\N^n \, : \, s(U)=F, \, |U|=i+1\}|
      \cdot \dim_{\kk} H^{\ell-1}\local{\Delta}{F} \\\label{Hochs}
& = & \sum_{F\in\Delta} \binom{i}{|F|-1} \cdot \dim_{\kk}
      H^{\ell-1}\local{\Delta}{F}.
\end{eqnarray}
For a general $m$, we prove the following (where we set
$\binom{a}{b}=0$ if $b<0$).

\begin{lemma} \label{ineq}
For every $0\leq m \leq d$, $\ell\leq d$, and $i\geq m$, 
$$%
  \dim_{\kk} \ker^\ell_{m,i} \leq \sum_{F\in\Delta}
  \binom{i-m}{|F|-m-1}\cdot \dim_{\kk}  H^{\ell-1}\local{\Delta}{F}.
$$
\end{lemma}

We defer the proof to the end of the section.  Using Lemma~\ref{ineq},
we can say even more.

\begin{lemma} \label{equality}
For every $0\leq m \leq d$, $\ell\leq d$, and $i\geq m$, 
$$%
  \dim_\kk \ker^\ell_{m,i} = \sum_{F\in\Delta} \binom{i-m}{|F|-m-1}
  \cdot \dim_{\kk} {H}^{\ell-1}\local{\Delta}{F}.
$$
Moreover, the map $\oplus_{i\geq m+1} \ker^\ell_{m,i} \stackrel{\cdot
\theta_{m+1}} {\longrightarrow} \oplus_{i\geq m+1}\ker^\ell_{m,i-1}$,
is a surjection.
\end{lemma}
\begin{proof}%
We prove the statement on the dimension of $\ker^\ell_{m,i}$ by
induction on $m$.  For $m=0$ (and any $\ell, i \geq 0$), this is
eq.~(\ref{Hochs}).  For larger $m$, we notice that the restriction of
$(\cdot \theta_{m+1})^{\ast}$ to $\ker^\ell_{m,i}$ is a linear map
from $\ker^\ell_{m,i}$ to $\ker^\ell_{m,i-1}$, whose kernel is
$\ker^\ell_{m+1,i}$.  Thus for $i\geq m+1$,
\begin{eqnarray*}
\dim_\kk \ker^\ell_{m+1,i}& \geq & 
\dim_\kk \ker^\ell_{m,i}-\dim \ker^\ell_{m,i-1}\\
 &=&
\sum_{F\in\Delta}
\left[\binom{i-m}{|F|-m-1}-\binom{i-1-m}{|F|-m-1}\right] \cdot
 \dim_\kk H^{\ell-1}\local{\Delta}{F} \\
 &=&\sum_{F\in\Delta}
\binom{i-(m+1)}{|F|-(m+1)-1}\cdot \dim_\kk {H}^{\ell-1}\local{\Delta}{F}.
\end{eqnarray*}
The second step in the above computation is by the inductive
hypothesis.  Comparing the resulting inequality to that of
Lemma~\ref{ineq} shows that this inequality is in fact equality, and
hence that the map $(\cdot\theta_{m+1})^\ast \, : \, \ker^\ell_{m,i}
\to \ker^\ell_{m,i-1}$ is surjective for $i\geq m+1$.
\end{proof}

Lemma~\ref{equality} allows us to get a handle on
$H^\ell_{\m}(\kk[\Delta]/(\theta_1, \ldots, \theta_m))_{-i}$ at least
for $\ell,i>0$.

\begin{theorem} \label{isomorphism}%
For $0\leq m \leq d$ and $0<\ell\leq d-m$, there is a graded
isomorphism of modules
$$%
  \bigoplus_{i\geq 1} H^\ell_\m\big(\kk[\Delta]/(\theta_1,\ldots,
  \theta_m)\big){}_{-i} \cong \bigoplus_{i\geq1} \ker^{\ell+m}_{m,i+m-1}
$$
compatible with the direct sum, hence
$H^\ell_\m(\kk[\Delta]/(\theta_1,\ldots, \theta_m))_{-i} \cong
\ker^{\ell+m}_{m,i+m-1}$.  Thus, 
$$%
  \dim_\kk H^\ell_\m\big(\kk[\Delta]/(\theta_1,\ldots,
  \theta_m)\big){}_{-i}= \sum_{F\in\Delta} \binom{i-1}{|F|-m-1}\cdot
  \dim_\kk {H}^{\ell+m-1}\local{\Delta}{F}  \quad \mbox{for $\ell,i>0$}.
$$ 
\end{theorem}

\begin{proof}%
The proof is by induction on $m$, with the $m=0$ case being evident.
For larger~$m$, we want to mimic the proof given in
Section~\ref{isolated sings}.  One obstacle to this approach is that
the map $\cdot\theta_{m+1}: \kk[\Delta]/(\theta_1,\ldots, \theta_m)\to
\kk[\Delta]/(\theta_1,\ldots, \theta_m)=:\M[m]$ might not be injective
anymore.  However, a ``prime avoidance'' argument together with the
genericity assumption on $\theta_{m+1}$ implies that
$\cdot\theta_{m+1}:\M[m]/H^0_\m(\M[m]) \to \M[m]/H^0_\m(\M[m])$ is
injective, and hence one has the corresponding long exact sequence in
local cohomology; see, for instance, \cite[Chapter~3]{Eis} for details
on ``prime avoidance'' arguments.  Since $H^0_\m(\M[m])$ has Krull
dimension 0, modding $H^0_\m(\M[m])$ out does not affect $H^\ell_\m$
for $\ell>0$, so that the part of this long exact sequence
corresponding to $\ell,i>0$ can be rewritten as
\begin{eqnarray*}
\bigoplus_{i\geq 1} H^\ell_\m(\M[m])_{-(i+1)} 
 &\stackrel{(\cdot \theta_{m+1})^\ast}{\longrightarrow}&
\bigoplus_{i\geq 1} H^\ell_\m(\M[m])_{-i}  \longrightarrow
\bigoplus_{i\geq 1} H^\ell_\m(\M[m+1])_{-i} \\  
&\stackrel{\delta}{\longrightarrow}& 
\bigoplus_{i\geq 1} H^{\ell+1}_\m(\M[m])_{-(i+1)} 
\stackrel{(\cdot \theta_{m+1})^\ast}{\longrightarrow} 
\bigoplus_{i\geq 1} H^{\ell+1}_\m(\M[m])_{-i}.
\end{eqnarray*}
By the inductive hypothesis combined with Lemma~\ref{equality}, the
leftmost map in this sequence is surjective.  Hence the module
$\oplus_{i\geq 1} H^\ell_\m(\M[m+1])_{-i}$ is isomorphic to the kernel
of the rightmost map.  Applying the inductive hypothesis to the last
two entries of the sequence then implies the following isomorphism of
modules:
\begin{equation*}%
  \bigoplus_{i\geq 1} H^\ell_\m(\M[m+1])_{-i} \cong \bigoplus_{i\geq
  1} \ker \big((\cdot \theta_{m+1})^\ast \, : \, \ker^{\ell+m+1}_{m,
  i+m} \to \ker^{\ell+m+1}_{m, i+m-1}\big) = \bigoplus_{i\geq 1}
  \ker^{\ell+m+1}_{m+1, i+m}.\qedhere
\end{equation*}
\end{proof}

Theorem~\ref{main} now follows easily from Theorem~\ref{isomorphism}.

\begin{proofof}{Theorem~\ref{main}}
Since $\kk[\Delta]$ is a finitely-generated algebra,
$H^0_\m\big(\kk[\Delta]/(\theta_1,\ldots, \theta_m)\big)$ has Krull
dimension zero, and hence is a finite-dimensional vector space for any
simplicial complex $\Delta$.  So we only need to care about
$H^\ell_\m$ for $\ell>0$.  As $\binom{i-1}{|F|-m-1}>0$ for all $i\gg
0$ and $|F|>m$, Theorem~\ref{isomorphism} implies that
$\kk[\Delta]/(\theta_1,\ldots, \theta_m)$ is a ring with finite local
cohomology if and only if for all faces $F\in\Delta$ of size larger
than $m$ and all $\ell+m<d$, the cohomology
${H}^{\ell+m-1}\local{\Delta}{F}$ vanishes.  Given that
${H}^{\ell+m-1}\local{\Delta}{F}$ is isomorphic to $
\tilde{H}^{\ell+m-1-|F|}(\lk F; \kk)$ (see e.g.\
\cite[Lemma~1.3]{Grabe}), this happens if and only if each such $F$ is
nonsingular.
\end{proofof}

It remains to verify Lemma~\ref{ineq}.  For its proof we use the
following notation.  Fix $m, \ell >0$, and $i\geq m$.  For $r\in \{i,
i+1\}$, let
$$%
  \V_r:=\{U\in\N^n \, : \, |U|=r, \, s(U)\in\Delta\},
$$
and for $F\in\Delta$, let $\V_{r,F}:= \{u\in\V_r \, : \, s(U)=F\}$.
If $F = \{f_1 < \cdots < f_j\} \in \Delta$ where $j > m$, then set
$\W_{r, F} := \{U=(u_1,\ldots, u_n)\in \V_{r,F} \, : \, u_{f_s}=1
\mbox{ for } 1\leq s\leq m \}$.  Observe that $\W_{i+1, F}$ is a
subset of $\V_{i+1, F}$ of cardinality $\binom{i-m}{|F|-m-1}$.

For $G\subseteq F\in\Delta$ let $\Phi_{F,G}$ denote the map $\inc^\ast
: H^{\ell-1}\local{\Delta}{F} \to H^{\ell-1}\local{\Delta}{G}$.  Thus,
$\Phi_{F,G}$ is the identity map if $F=G$.  Using Theorem~\ref{Grabe},
we identify $H^\ell_\m(\kk[\Delta])_{-r}$ with $\bigoplus_{U\in\V_{r}}
H^{\ell-1}\local{\Delta}{s(U)}$, and for $z\in
H^\ell_\m(\kk[\Delta])_{-r}$ we write $z=(z_U)_{U\in\V_{r}}$, where
$z_U\in H^{\ell-1}\local{\Delta}{s(U)}$ is the $(-U)$-th component of
$z$.  Since $\theta_p=\sum_{t=1}^n a_{t,p}x_t$, Theorem~\ref{Grabe}
yields that for such $z$, $r=i+1$, and $T\in \V_i$,
\begin{equation} \label{mult}
(\theta_p z)_T = \sum_{\{t \, : \, T+e_t\in\V_{i+1}\}}
    a_{t,p} \cdot \Phi_{s(T+e_t), s(T)}(z_{T+e_t}).
\end{equation}

\begin{proofof*}{Lemma~\ref{ineq}}
Since $|\W_{i+1, F}|=\binom{i-m}{|F|-m-1}$ for all $F\in\Delta$, to
prove that $\dim_{\kk} \ker^\ell_{m,i} \leq \sum_{F\in\Delta}
\binom{i-m}{|F|-m-1}\cdot \dim_{\kk} H^{\ell-1}\local{\Delta}{F}$, it
is enough to verify that for $z,z'\in \ker^\ell_{m,i}$,
$$%
  z_U=z'_U \ \text{ for all }  F\in\Delta \text{ and all }
  U\in\W_{i+1, F} \quad \Longrightarrow \quad z=z',
$$ 
or equivalently (since $\ker^\ell_{m,i}$ is a $\kk$-space) that for
$z\in \ker^\ell_{m,i}$,
\begin{equation} \label{initial}
  z_U=0 \ \mbox{ for all }  F\in\Delta \mbox{ and all } U\in\W_{i+1,
  F} \quad \Longrightarrow \quad z=0.
\end{equation}
To prove this, fix such a $z$.  From eq.~(\ref{mult}) and the
definition of $\ker^\ell_{m,i}$, it follows that
\begin{equation} \label{zero}
\sum_{\{t \, : \, T+e_t\in\V_{i+1}\}}
  a_{t,p} \cdot \Phi_{s(T+e_t), s(T)}(z_{T+e_t}) =0 
  \quad \forall \ 1\leq p \leq m \mbox{ and } \forall \  T\in \V_i.
\end{equation}
For a given $T\in\V_i$, we refer to the $m$ conditions imposed on $z$
by eq.~(\ref{zero}) as ``the system defined by~$T$'', and denote this
system by~$\S_T$.

Define a partial order, $\succ$, on $\V_{i}$ as follows: $T'\succ T$
if either $|s(T')|>|s(T)|$, or $s(T')=s(T)$ and the last non-zero
entry of $T'-T$ is positive.  To finish the proof, we verify by a
descending (with respect to~$\succ$) induction on $T\in\V_i$, that
$z_{T+e_t}=0$ for all $t$ with $T+e_t\in\V_{i+1}$.  For $T\in\V_i$,
there are two possible cases (we assume that $s(T)=F=\{f_1< \ldots
<f_j\}$).

\begin{case}{Case 1.}
$T\in\W_{i,F}$.  Then for each $t\notin F$, either
$F':=F\cup\{t\}\notin\Delta$, in which case $T+e_t\notin\V_{i+1}$, or
$T+e_t\in\W_{i+1, F'}$, in which case $z_{T+e_t}=0$ by
eq.~(\ref{initial}).  Similarly, if $t=f_r$ for some $r>m$, then
$T+e_t\in \W_{i+1, F}$, and $z_{T+e_t}=0$ by (\ref{initial}).
Finally, for any $t\in F$, $s(T+e_t)=s(T)=F$, and so $\Phi_{s(T+e_t),
s(T)}$ is the identity map.  Thus the system $\S_T$ reduces to $m$
linear equations in $m$ variables:
\begin{equation} \label{system}
  \sum_{r=1}^m a_{f_r, p}\cdot z_{T+e_{f_r}} =0  
  \quad \forall  \  1\leq p \leq m.
\end{equation}
Since the matrix $(a_{f_r,p})_{1\leq r,p \leq m}$ is nonsingular,
$(z_{T+e_t}=0 \mbox{ for all }t)$ is the only solution of~$\S_T$.
\end{case}

\begin{case}{Case 2.}
$T\notin \W_{i,F}$, and so $T_{f_s}\geq 2$ for some $s\leq m$.  Then
for any $t\notin F$, $T':=T-e_{f_s}+e_t \succ T$, as $T'$ has a larger
support than $T$, and $T'+e_{f_s}=T+e_t$.  Hence
$z_{T+e_t}=z_{T'+e_{f_s}}=0$ by the inductive hypothesis on $T'$.
Similarly, if $t\in F-\{f_1, \ldots, f_m\}$, then $t>f_s$, and so
$T'':= T-e_{f_s}+e_t \succ T$.  As $T''+e_{f_s}=T+e_t$, the inductive
hypothesis on $T''$ imply that $z_{T+e_t}=0$.  Thus, as in Case 1,
$\S_T$ reduces to system (\ref{system}), whose only solution is
trivial.\qed
\end{case}
\end{proofof*}

\section{Complexes Cohen--Macaulay in a fixed codimension}\label{codim}

Our goal for this section, in Theorem~\ref{t:singdim}, is to rephrase
the notion of singularity dimension of a simplicial complex~$\Delta$
as a local geometric condition on the (spectrum of) the face
ring~$\kk[\Delta]$, analogous to---and generalizing---the
Cohen--Macaulay condition.

\begin{defn} A simplicial complex
$\Delta$ is \emph{CM of dimension~$i$ along $F$} if $\lk F$ is
Cohen--Macaulay of dimension~$i$.
\end{defn}

\begin{remark}
If $\Delta$ has dimension $d-1$, then for $\Delta$ to be CM of
dimension $d - 1 - |F|$ along $F$ is stronger than for $F$ to be a
nonsingular face of~$\Delta$.  For example, nonsingularity
of~$\nothing$ means only that $\Delta$ is a homology-wedge-of-spheres,
while CM of dimension $d-1$ along~$\nothing$ implies nonsingularity of
all faces of~$\Delta$.
\end{remark}

\begin{defn}
A simplicial complex $\Delta$ of dimension~$r$ is \emph{CM in
codimension~$c$} if $\Delta$ is CM of dimension $c - 1$ along every
face of dimension $r - c$, or if $c > r+1$ and $\Delta$ is
Cohen--Macaulay.
\end{defn}

\begin{remark}\label{r:pure}
If $\Delta$ is pure, then it would be enough to require that $\lk F$
is Cohen--Macaulay for every face~$F$ of dimension $r-c$; but if
$\Delta$ has dimension~$r$ and is not pure, then it is possible for
$\lk F$ to be Cohen--Macaulay without $\Delta$ being CM of dimension
$r - |F|$ along~$F$ (such a face~$F$ is singular in~$\Delta$ unless
$\lk F$ is $\kk$-acyclic).  This occurs, for example, when $F$ is a
facet of dimension $< r$, in which case $\lk F = \{\nothing\}$ is
Cohen--Macaulay of dimension~$-1$.
\end{remark}

\begin{proposition}
Fix a simplicial complex $\Delta$ of dimension~$r$.  The singularity
dimension of~$\Delta$ is the minimum~$i$ such that $\Delta$ is CM in
codimension $r-i-1$.
\end{proposition}
\begin{proof}
If $\Delta$ has singularity dimension~$s$, then $\Delta$ is CM of the
appropriate dimension along every face of dimension $\geq s+1$ by
definition and Reisner's criterion (if~$\lk F$ had too small
dimension, then the link of any facet containing~$F$ would give this
fact away, by Remark~\ref{r:pure}); thus $\Delta$ is CM in codimension
$r-s-1$.  On the other hand, if $\Delta$ is CM in codimension $r-i-1$
then every face of dimension $i+1$ is nonsingular by definition.
\end{proof}

\begin{corollary}\label{c:singdim}
Fix a simplicial complex $\Delta$ of dimension~$r$.  The singularity
dimension of~$\Delta$ is less than~$m$ if and only if $\Delta$ is CM
in codimension $r-m$.
\end{corollary}

This corollary and the next lemma form the bridge to general
commutative algebra.

\begin{lemma}\label{l:CMalong} A simplicial complex
$\Delta$ is CM of dimension~$i$ along~$F$ if and only if the
localization $\kk[\Delta]_{P_F}$ at the prime ideal $P_F = (x_j : j
\notin F)$ of~$F$ is a Cohen--Macaulay ring of Krull dimension~$i+1$.
\end{lemma}
\begin{proof}
Localizing $\kk[\Delta]$ by inverting the variables indexed by~$F$
results in $\kk[\lk F][x_i^{\pm1} : i \in F]$.  The lemma is
straightforward from this and Reisner's criterion
\cite[Theorem~5.53]{cca}.
\end{proof}

\begin{defn}\label{d:module}
A module $M$ over a noetherian ring is \emph{CM of dimension~$i$
locally at~$\pp$} if the localization $M_\pp$ at the prime ideal~$\pp$
is a Cohen--Macaulay module of Krull dimension~$i$.  If $\dim M = d$,
then $M$ is \emph{CM in codimension~$c$} if $M$ is CM of dimension~$c$
locally at every prime~$\pp$ of dimension $d-c$ in the support of~$M$,
or if $c > d$ and $M$ is Cohen--Macaulay.
\end{defn}

\begin{theorem}\label{t:singdim}
A simplicial complex $\Delta$ of dimension~$r$ is CM in codimension
$c$ if and only if\/ $\kk[\Delta]$ is~CM in codimension~$c$.  In
particular, the singularity dimension of~$\Delta$ is less than~$m$ if
and only if\/ $\kk[\Delta]$ is CM in codimension $r-m$.
\end{theorem}
\begin{proof}
Use Lemma~\ref{l:CMalong} for the first sentence and add
Corollary~\ref{c:singdim} for the second.
\end{proof}

\begin{remark}\label{r:module}
Definition~\ref{d:module} is related to, but different from, Serre's
condition~$S_c$.  Both conditions can be interpreted as stipulating
the Cohen--Macaulay condition near the prime ideals of certain large
irreducible subvarieties, but $S_c$ requires a module to be
Cohen--Macaulay near every prime whose local ring has \emph{depth} at
least~$c$, whereas Definition~\ref{d:module} requires a module to be
Cohen--Macaulay at every prime whose local ring has \emph{dimension}
at least~$c$.  These two conditions manifest differently in local
cohomology: a module~$M$ over a regular local ring with maximal
ideal~$\mm$ satisfies Serre's condition~$S_c$ when its local
cohomology $H^i_\mm(M)$ vanishes in cohomological degrees $i < c$,
whereas $M$ is~CM in codimension~$c$ when the Matlis dual of its local
cohomology satisfies $\dim H^i_\mm(M)^\vee < \dim M - c$ for~$i < \dim
M$ (apply local duality to Corollary~\ref{c:regular} below).  When $M$
is a ring, this means that the homology of its dualizing complex has
dimension $< \dim M - c$ except at the end where it has
dimension~$\dim M$.  In short, Serre's condition bounds the
cohomological degrees of the nonvanishing local cohomology modules,
whereas CM in a fixed codimension bounds their Krull dimensions.
See~\cite{murai-terai} for recent work on combinatorial implications
of Serre's conditions on face rings.
\end{remark}

\section{Modules Cohen--Macaulay in a fixed codimension}\label{modules}

With Theorem~\ref{t:singdim} in mind, Theorem~\ref{main} is a
statement about the behavior, under quotients by generic linear forms,
of certain graded rings that are Cohen--Macaulay in a fixed codimension.
The end result of this section, Theorem~\ref{t:algebraic},
demonstrates the rather general nature of this statement: it holds for
any finitely generated graded module over any finitely generated
standard $\ZZ$-graded $\kk$-algebra, where $\kk$ is an infinite field.

Since the point of the developments of this section 
is to ignore certain modules of
finite length, the main result is best proved in the language of
sheaves on projective schemes, which we do in Theorem~\ref{t:codim}.
In that formulation, it is especially easy to see how the developments
here relate general Bertini-type theorems in modern algebraic
geometry; see Remark~\ref{r:general}.

We begin by rephrasing Definition~\ref{d:module} in geometric terms.
Recall that the \emph{dimension} of a point~$\pp$ in a scheme is the
supremum of the lengths~$\ell$ of chains $\pp = \pp_\ell > \cdots >
\pp_1 > \pp_0$ where $\pp_i > \pp_j$ means that $\pp_j$ lies in the
closure of~$\pp_i$.  We use the convention that a sheaf or module is
zero if and only if its Krull dimension is negative.

\begin{defn}\label{d:codim}
A coherent sheaf $\FF$ on a noetherian scheme with structure
sheaf~$\OO$ is \emph{CM of dimension~$i$ locally at~$\pp$} if the germ
$\FF_\pp$ is a Cohen--Macaulay $\OO_\pp$-module of Krull
dimension~$i$.  If $\dim\FF = r$, then $\FF$ is \emph{CM in
codimension~$c$} if $\FF$ is CM of dimension~$c$ locally at every
point~$\pp$ of dimension $r-c$ in the support of~$\FF$, or if $c > r$
and $\FF$ is Cohen--Macaulay.
\end{defn}

\begin{proposition}\label{p:regular}
A coherent sheaf $\FF$ of dimension~$r$ on a regular scheme~$X$ of
dimension~$\delta$ is CM in codimension~$c$ if and only if $\dim
\Ext^i_X(\FF,\OO_X) < r - c$ for $i > \delta - r$.
\end{proposition}
\begin{proof}
The sheaf $\Ext^i_X(\FF,\OO_X)$ has dimension $< r - c$ if and only if
its germ at every point~$\pp$ of dimension $r - c$ is zero.  The germ
in question is $\Ext^i_X(\FF,\OO_X)_\pp =
\ext^i_{\OO_{X,\pp}}(\FF_\pp,\OO_{X,\pp})$ because $\FF$ is coherent
and $X$ is noetherian.  At any point~$\pp$ in the support of~$\FF$,
the germ $\FF_\pp$ at~$\pp$ has dimension at most $r - \dim\pp$ and
depth equal to $\delta - \dim\pp - \max\{i :
\ext^i_{\OO_{X,\pp}}(\FF_\pp,\OO_{X,\pp}) \neq 0\}$, the latter by
\cite[Exercise~3.1.24]{BrHe}.  The dimension and depth equal~$c$ for
all points~$\pp$ of dimension $r - c$ if and only if the maximum is
$\delta - r$ for all such~$\pp$, and the first sentence of the proof
implies that this is equivalent to the desired dimension condition
on~$\Ext^i_X(\FF,\OO_X)$.
\end{proof}

The following algebraic version is immediate from the affine case of
Proposition~\ref{p:regular}.

\begin{corollary}\label{c:regular}
A finitely generated module $M$ of dimension~$d$ over $S =
\kk[x_1,\ldots,x_n]$ is CM in codimension~$c$ if and only if
$\dim\ext^i_S(M,S) < d - c$ for $i > n - d$.\qed
\end{corollary}

Applying the above corollary to the second sentence of
Theorem~\ref{t:singdim} yields the following important consequence,
noting that $c = r - m \iff d - c = m + 1$ when $r = d - 1$.

\begin{corollary}\label{c:ext}
A simplicial complex $\Delta$ of dimension $r = d-1$ has singularity
dimension less than~$m$ if and only if\/ $\dim\ext^i_S(\kk[\Delta],S)
\leq m$ for $i > n - d$.\qed
\end{corollary}

\begin{remark}
Since local cohomology and $\ext$ into~$S$ are Matlis dual by local
duality, and since the Krull dimension of a module $M$ equals the
order to which $\lambda = 1$ is a pole of the $\ZZ$-graded Hilbert
series $F(M, \lambda)$ of $M$, Corollary~\ref{c:ext} is equivalent to
Corollary~\ref{comb-char}.
\end{remark}

Another consequence of Proposition~\ref{p:regular} is a
characterization of finite local cohomology.

\begin{corollary}\label{c:algebraic}
A finitely generated standard $\ZZ$-graded $S$-module~$M$ of
dimension~$d$ has finite local cohomology if and only if $M$ is CM in
codimension~$d-1$.
\end{corollary}
\begin{proof}
The module $\ext^i_S(M,S)$ is the graded Matlis dual of
$H^{n-i}_\mm(M)$ by local duality.  Thus $M$ has finite local
cohomology if and only if $\dim \ext^i_S(M,S) \leq 0$ for $i > n - d$.
\end{proof}

Next is the main result of this section, a Bertini-type theorem for CM
in codimension~$c$.

\begin{theorem}\label{t:codim}
Fix a coherent sheaf $\FF$ on~$\PP^n$ and a linear form $\theta$
on~$\PP^n$ vanishing on a hyperplane~$H$.  Assume that $\theta$ is
nonzero at every associated point of\/ $\bigoplus_i
\Ext^i_{\PP^n}(\FF,\OO_{\PP^n})$.  For $c < \dim\FF$, the sheaf $\FF$
is CM in codimension~$c$ if and only if its restriction $\FF|_H$ is.
If~$\FF$ is Cohen--Macaulay (this is the case of $c \geq \dim\FF$),
its restriction $\FF|_H$ is also Cohen--Macaulay.
\end{theorem}
\begin{proof}
As the final sentence is well-known for Cohen--Macaulay sheaves, we fix
$c <\nolinebreak \dim\FF$.  Let $\EE_\spot$ be a resolution of~$\FF$
by locally free sheaves, and write $\EE^\spot$ for the dual complex of
sheaves. Thus the cohomology of~$\EE^\spot$ is
$\Ext^\spot_{\PP^n}(\FF,\OO_{\PP^n})$.  The restriction $\EE^\spot|_H$
to~$H$ is the cokernel of the sheaf homomorphism $\EE^\spot(-1) \to
\EE^\spot$ induced by~$\theta$.  Its cohomology is
$\Ext^\spot_H(\FF|_H,\OO_H)$; this is by the genericity hypothesis
on~$\theta$, which also guarantees that the cohomology sheaves of
$\EE^\spot|_H$ are obtained by restricting those of~$\EE^\spot$
to~$H$.  The dimensions of the cohomology sheaves decrease by
precisely~$1$ (recall our negative dimension convention) by genericity
again and the Hauptidealsatz; geometrically: $H$~is transverse to
every component in the support of each $\Ext$ sheaf.  Consequently,
$\dim\Ext^\spot_{\PP^n}(\FF,\OO_{\PP^n}) < \dim\FF - c \iff
\dim\Ext^\spot_H(\FF|_H,\OO_H) < \dim\FF_H - c$, where the
``$\Leftarrow$'' direction requires the hypothesis $\dim\FF - c > 0$.
Now apply Proposition~\ref{p:regular} to
$\Ext^\spot_{\PP^n}(\FF,\OO_{\PP^n})$ and $\Ext^\spot_H(\FF|_H,\OO_H)$
to complete the proof.
\end{proof}

\begin{proposition}\label{p:sheafify}
Fix a graded $S$-module~$M\hspace{-.2pt}$ with associated sheaf $\FF$
on projective space~$\PP^{n-1}$.  If $c \leq \dim\FF$,
then $M$ is CM in codimension~$c$ if and only if $\FF$ is CM in
codimension~$c$.
\end{proposition}
\begin{proof}
By functoriality, $\Ext^i_X(\FF,\OO_X)$ is the sheaf associated to
$\ext^i_S(M,S)$, where $X = \PP^{n-1}$, hence the the result is a
straightforward consequence of Proposition~\ref{p:regular} and
Corollary~\ref{c:regular}.
\end{proof}

The following algebraic version of Theorem~\ref{t:codim} is slightly
less clean for $c \approx \dim M$, owing to the possibility of the
maximal ideal~$\mm$ as an associated prime of $M$ and its $\ext$
modules.

\begin{theorem}\label{t:algebraic}
Fix a graded $S$-module $M$ and a linear form~$\theta$ not in any
associated prime $\neq \mm$ of\/ $\bigoplus_i\ext^i_S(M,S)$.
If $c < \dim M - 1$ then $M/\theta M$ is CM in codimension~$c$ if and
only if $M$~is.  If $M$ has finite local cohomology or is
Cohen--Macaulay then so is~$M/\theta M$, accordingly.
\end{theorem}
\begin{proof}
If $c < \dim M - 1$, then the result follows from
Theorem~\ref{t:codim} and Proposition~\ref{p:sheafify}.  If $c \geq
\dim M - 1$ and $M$ is CM in codimension~$c$, then 
(by Corollary~\ref{c:algebraic}) the sheaf
on~$\PP^{n-1}$ associated to~$M$ is Cohen--Macaulay, and so the same is
true of the sheaf associated to~$M/\theta M$ by Theorem~\ref{t:codim},
but $M/\theta M$ might itself have (nonzero) finite local cohomology.
\end{proof}

Finally, it remains to see that Theorem~\ref{t:algebraic} does indeed
directly generalize Theorem~\ref{main}.

\begin{proofof}{Theorem~\ref{main}}
Let $\Delta$ have dimension $r = d - 1$, so $\kk[\Delta]$ has
dimension $d$.  The quotient $\kk[\Delta]/(\theta_1,\ldots,\theta_m)$
modulo a sequence of~$m$ generic linear forms has finite local
cohomology if and only if it is CM in codimension $(d - m) - 1 = r -
m$, by Corollary~\ref{c:algebraic}.  This occurs if and only if
$\kk[\Delta]$ itself is CM in codimension $r - m$, by
Theorem~\ref{t:algebraic}.  The desired result follows from the
characterization of singularity dimension in Theorem~\ref{t:singdim}.
\end{proofof}

\begin{remark}\label{r:general}
The first result in the general direction of Theorems~\ref{t:codim}
and~\ref{t:algebraic} was proved by Flenner for subschemes of spectra
of local rings \cite[Satz~3.3]{flenner}.
It is a special case of Spreafico's very general Bertini-type theorem
for local geometric conditions satisfying certain natural axioms
\cite[Corollary~4.3]{spreafico}; for the purposes here, the reader can
check that CM in codimension~$c$ is a valid choice for Spreafico's
``property $\PP$''.
However, Flenner and Spreafico proved their results for subschemes,
not coherent sheaves, and the ``if'' direction of
Theorems~\ref{t:codim} and~\ref{t:algebraic} does not fall under
Spreafico's framework: one cannot conclude that a variety is (for
instance) smooth by knowing that a generic hyperplane section is
smooth.  In contrast, the constrained version of the Cohen--Macaulay
condition furnished by Definition~\ref{d:codim} lifts from hyperplane
sections because a hyperplane in projective space intersects every
subvariety of positive dimension and---generically, at least---reduces
its dimension by~$1$, the operative subvariety being the locus where a
sheaf fails to be CM in codimension~$c$.  The failure of hyperplanes
to meet subschemes of dimension~$0$ accounts for the failure of the
``if'' direction when the codimension~$c$ is large.
\end{remark}

\section{Squarefree modules modulo systems of parameters}\label{sqfree}

In Section~\ref{proof}, we derived the simplicial Theorem~\ref{main}
from the simplicial calculations of Hilbert series of local cohomology
modules that constitute Theorem~\ref{isomorphism}.  Having taken the
former to its natural generality in Sections~\ref{codim}
and~\ref{modules}, we now generalize the latter in
Theorem~\ref{t:sqfree}, where $H^{\ell+m}_F(M) =
H^{\ell+m-1}\local{\Delta}{F}$ when $M = \kk[\Delta]$ by
Theorem~\ref{Grabe}.  It~is possible to imagine that generalizations
to arbitrary $\NN^n$-graded modules could exist, but we limit
ourselves here to the class of squarefree modules, introduced by
Yanagawa \cite{yanagawa}.
\begin{defn}
A $\ZZ^n$-graded module over $S = \kk[x_1,\ldots,x_n]$ is
\emph{squarefree} if its generators and relations all lie in
\emph{squarefree degrees}, meaning those in $\{0,1\}^n
\subset\nolinebreak \ZZ^n$.
\end{defn}
The point of squarefree modules is that, like face
rings~$\kk[\Delta]$, all of the information can be recovered easily
from the finite-dimensional vector subspace in squarefree degrees.

One of the fundamental conclusions of Section~\ref{proof} is that the
entirety of the interesting combinatorial information encoded by the
local cohomology of~$\kk[\Delta]$ is distilled into the finite length
modules in $\ZZ$-graded degrees zero and above that remain after
applying Theorem~\ref{isomorphism} enough times.  Other types of
reductions to vector spaces of finite-dimension are known for the
local cohomology of~$\kk[\Delta]$, and indeed of any squarefree
module~$M$: one way to phrase it is that the $\ZZ^n$-graded Matlis
dual $H^i_\mm(M)^\vee = \ext^{n-i}_S(M,\omega_S)$ is a squarefree
module \cite[Theorem~2.6]{yanagawa}, where the \emph{canonical module}
$\omega_S$ is the rank~$1$ free module generated in $\ZZ^n$-graded
degree $(1,\ldots,1)$.  (An analogous statement holds for the local
cohomology of arbitrary $\NN^n$-graded modules: the \v Cech hull
\cite[Definition~2.7]{localDuality} recovers the local cohomology from
an a~priori bounded set of degrees; see the proof of
\cite[Lemma~3.10]{cmMonomial}.)

In fact, the squarefreeness of these $\ext$ modules is the key to
Theorem~\ref{t:sqfree}, via the general enumerative statement about
arbitrary squarefree modules in Proposition~\ref{p:sqfree}.  First, we
need a lemma, whose proof relies on a key feature (the squarefree
filtration) of squarefree modules.

\begin{lemma}\label{l:tor}
If $M$ is a squarefree $S$-module and $R = S/\Theta S$ for a sequence
$\Theta = \theta_1,\ldots,\theta_m$ of generic linear forms, then
$\tor_t^S(M,R)_i = 0$ whenever $t > 0$ and $i > m$.
\end{lemma}
\begin{proof}
For each squarefree vector~$F$, set $\xx^F = \prod_{j \in F} x_j$ and
write $\kk F$ for the $\ZZ^n$-graded $S$-module $\xx^F\kk[x_j : j \in
F] = \xx^F S/(x_j : j \notin F)$ generated in degree~$F$.  As a
standard $\ZZ$-graded module, $\kk F$ is a polynomial ring in $|F|$
many variables generated in degree~$|F|$.

The squarefree module~$M$ has a filtration
\begin{equation}\label{e:M}
  0 = M_0 \subset M_1 \subset \cdots \subset M_k = M
\end{equation}
such that each quotient $M_j/M_{j-1}$ is isomorphic, as a
$\ZZ^n$-graded module, to $\kk F$ for some squarefree vector~$F$
\cite[Proposition~2.5]{yanagawa}.  If the filtration has length~$1$,
and so $M = \kk F$, then the module $\tor_t^S(\kk F,R)$ vanishes for
$t > 0$ whenever $|F| \geq m$. This is because $\Theta$ is a regular
sequence on~$\kk F$ in that case.  When $|F| < m$, on the other hand,
$\tor_t^S(\kk F,R) \cong \tor_t^{S'}\big(\kk(-|F|),R\big)$, where $S'
=S/(\theta_1,\ldots,\theta_{|F|})$ and $\kk(-|F|) =\kk F\otimes_S S'$.
Thus the $\tor$ in question is a direct sum of $\binom{m - |F|}{t}$
many copies of the residue field~$\kk$ in $\ZZ$-graded degree $|F| +
t$.  The desired result holds for $M = \kk F$ since $\binom{m -
|F|}{t} = 0$ as soon as $t > m - |F|$.

The general case follows by induction on the length of the
filtration~\eqref{e:M} by tensoring the exact sequence $0 \to M_{j-1}
\to M_j \to M_j/M_{j-1} \to 0$ with the Koszul complex on~$\Theta$.
\end{proof}

\begin{proposition}\label{p:sqfree}
Fix a squarefree module $E$ and a sequence $\Theta =
\theta_1,\ldots,\theta_m$ of generic linear forms.  For all $i \geq
0$, the vector space dimension of the $i$-th standard $\ZZ$-graded
piece of~$E$ is
$$%
  \dim_\kk E_i
  =
  \sum_{F\in\{0,1\}^n} \binom{i - 1}{|F| - 1} \cdot \dim_\kk E_F,
$$%
where $E_F$ is the piece of~$E$ in $\ZZ^n$-graded degree~$F$.  In
contrast, for all $i > m$,
$$%
  \dim_\kk(E/\Theta E)_i
  =
  \sum_{F\in\{0,1\}^n}\binom{i-m-1}{|F|-m-1}\cdot\dim_\kk E_F.
$$
\end{proposition}
\begin{proof}
The formula for $\dim_\kk E_i$ simply expresses the fact that $E$ and
its associated graded module for the squarefree filtration~\eqref{e:M}
have the same Hilbert function, which is the sum of the Hilbert
functions of the associated graded modules~$\kk F$; the binomial
coefficient counts the monomials of degree $i - |F|$ in (at most)
$|F|$ many variables.

The binomial coefficient in the formula for $\dim_\kk(E/\Theta E)_i$
counts the monomials of degree $i - |F|$ in $|F| - m$ variables, which
is the dimension of the $i$-th $\ZZ$-graded piece of $\kk F/\Theta \kk
F$.  It therefore suffices to show that the Hilbert functions of
$E/\Theta E$ and $\gr E/\Theta\gr E$ agree in degrees $> m$.  To see
this, tensor the short exact sequence $0 \to E_{j-1} \to E_j \to
E_j/E_{j-1} \to 0$ with $S/\Theta$ and note that it remains exact in
$\ZZ$-graded degrees $> m$ by Lemma~\ref{l:tor}.
\end{proof}

For the next result, assume that a system $\Theta =
\theta_1,\ldots,\theta_m$ of generic linear forms has been fixed, and
set $R = S/\Theta S$.  The standard $\ZZ$-graded canonical module
of~$S$ is $\omega_S = S(-n)$, whereas the $\ZZ$-graded canonical
module of~$R$ is $\omega_R = R(-(n-m)) = R \otimes_S \omega_S(m)$.

\begin{proposition}\label{p:ext}
For any standard $\ZZ$-graded $S$-module~$M$, the sheaves on
projective space determined by $R \otimes_S
\ext^\ell_S(M,\omega_S)(m)$ and\/ $\ext^\ell_R(R \otimes_S
M,\omega_R)$ are isomorphic; that is, the modules themselves are
isomorphic up to a kernel and cokernel supported at the maximal~ideal.
\end{proposition}
\begin{proof}
Reasoning as in the proofs of Theorem~\ref{t:codim} and
Proposition~\ref{p:sheafify} demonstrates that $R \otimes_S
\ext^\ell_S(M,S)$ and $\ext^\ell_R(R \otimes_S M,R)$ determine
isomorphic sheaves.  Beyond that, all that remains is fiddling with
the $\ZZ$-graded degree shifts.
\end{proof}

\begin{theorem}\label{t:sqfree}
Fix a squarefree $S$-module $M$ and a sequence $\Theta =
\theta_1,\ldots,\theta_m$ of generic linear forms.  If $H^j_F(M) :=
H^j_\mm(M)_{-F}$ is the $\ZZ^n$-graded piece of the $j$-th local
cohomology of~$M$ in the negative of the squarefree degree~$F$, then
for all $i \gg 0$,
$$%
  \dim_\kk H^\ell_\mm(M/\Theta M)_{-i}
  =
  \sum_F\binom{i-1}{|F|-m-1}\cdot\dim_\kk H^{\ell+m}_F(M).
$$
\end{theorem}
\begin{proof}
Using $\omega_S = x_1\cdots x_n S = S(-1,\ldots,-1)$ and an asterisk
for vector space duals, $H^{\ell+m}_\mm(M)_{-\alpha} \cong
\ext^{n-m-\ell}_S(M,\omega_S)_\alpha^*$ for all $\alpha \in \ZZ^n$ by
$\ZZ^n$-graded local duality \cite[Corollary~6.1]{localDuality}.  Thus
$H^j_F(M) \cong \ext^{n-j}_S(M,\omega_S)_F^*$ for all cohomological
degrees~$j$ and all squarefree vectors~$F$, and so
Proposition~\ref{p:sqfree} implies that the right-hand side of the
desired formula is the dimension of the $\ZZ$-graded piece
$R \otimes\ext^{n-m-\ell}_S(M,\omega_S)_{i+m}$.  The result follows from
Proposition~\ref{p:ext}, given that $H^\ell_\mm(M/\Theta M)_{-i} \cong
\ext^{n-m-\ell}_R(R \otimes M,\omega_R)_i^*$ by $\ZZ$-graded local
duality \cite[Theorem~3.6.19]{BrHe}.
\end{proof}

The special case of Theorem~\ref{t:sqfree} in which $M = \kk[\Delta]$
is slightly weaker than Theorem~\ref{isomorphism}, because of the
condition $i \gg 0$, but we believe this to be an artifact of the
proof.

\begin{conjecture}
The formula for $\dim_\kk H^\ell_\mm(M/\Theta M)_{-i}$ in
Theorem~\ref{t:sqfree} holds for all $i > 0$.
\end{conjecture}

\begin{remark}
The conjecture easily reduces to a more precise version of
Proposition~\ref{p:ext}: we need $R \otimes_S
\ext^\ell_S(M,\omega_S)(m)$ and $\ext^\ell_R(R \otimes_S M,\omega_R)$
to be isomorphic in all positive $\ZZ$-graded degrees when $M$ is a
squarefree module.  This probably follows from the full strength of
Lemma~\ref{l:tor} (only the case $t = 1$ was used in the proof of
Proposition~\ref{p:sqfree}) and a comparison between $\ext^\ell_R(R
\otimes_S M,\omega_R)$ and the cohomology of the total complex of
$\Hom_S(K_\spot \otimes F_\spot,\omega_S)(m)$, where $K_\spot =
K_\spot(\Theta)$ is the Koszul complex of $\theta_1,\ldots,\theta_m$
and $F_\spot$ is a resolution of~$M$ by graded free $S$-modules.
However, the derivation of Theorem~\ref{main} in Section~\ref{proof}
only requires the weakened version of Theorem~\ref{isomorphism} that
is an immediate consequence of Theorem~\ref{t:sqfree} as stated.
Therefore, since the conjecture is not central to our conclusions
regarding finite local cohomology, we leave it open.
\end{remark}


\begin{thebibliography}{99}

\bibitem{BrHe}
W.~Bruns and J.~Herzog,
\emph{Cohen--Macaulay rings},
Cambridge Studies in Advanced Mathematics, vol.~39, 
Cambridge University Press, Cambridge, 1993.

\bibitem{CST}
N.T.~Coung, P.~Schenzel, and N.V.~Trung,
\emph{Verallgemeinerte Cohen--Macaulay-Moduln},
Math.~Nachr.~\textbf{85} (1978), 57--73.

\bibitem{Eis}
D.~Eisenbud,
\emph{Commutative Algebra with a View Toward Algebraic Geometry},
Springer-Verlag, New York, 1995.

\bibitem{flenner}
H.~Flenner,
\emph{Die S\"atze von Bertini f\"ur lokale Ringe},
Math.~Ann.~\textbf{229} (1977), no.~2, 97--111.

\bibitem{GotoYuk}
S.~Goto and Y.~Takayama,
\emph{Stanley--Reisner ideals whose powers have finite length cohomologies},
Proc.~Amer.~Math.~Soc.~\textbf{135} (2007), 2355--2364.

\bibitem{Grabe}
H.-G.~Gr\"abe,
\emph{The canonical module of a Stanley--Reisner ring},
J.~Algebra \textbf{86} (1984), 272--281.

\bibitem{localDuality}
E.~Miller,
\emph{The Alexander duality functors and local duality with monomial support},
J.~Algebra \textbf{231} (2000), 180--234.

\bibitem{cmMonomial}
E.~Miller,
\emph{Topological Cohen--Macaulay criteria for monomial ideals},
in \emph{Combinatorial aspects of commutative algebra}
(Viviana Ene and Ezra Miller, eds.),
Contemporary Mathematics, Vol.~502,
American Mathematical Society, Providence, RI, 2009,
137--156.  \textsf{arXiv:math.AC/0809.1458}

\bibitem{cca}
E.~Miller and B.~Sturmfels,
\emph{Combinatorial Commutative Algebra},
Graduate Texts in Mathematics, vol.~227, Springer--Verlag, New York, 2005.

\bibitem{murai-terai}
S.~Murai and N.~Terai,
\emph{$h$-vectors of simplicial complexes with Serre's conditions},
preprint.  \textsf{arXiv:math.AC/0912.1089}

\bibitem{Reisner}
G.~Reisner,
\emph{Cohen--Macaulay quotients of polynomial rings},
Adv.~in Math.~\textbf{21} (1976), 30--49.

\bibitem{SchSch}
U.~Sch\"afer and P.~Schenzel,
\emph{Dualizing complexes of affine semigroup rings},
Trans. Amer. Math. Soc. \textbf{322} (1990), no.~2, 561--582.

\bibitem{Sch}
P.~Schenzel,
\emph{On the number of faces of simplicial complexes and the 
purity of Frobenius},
Math.~Z.~\textbf{178} (1981), 125--142.

\bibitem{spreafico}
M.~L.~Spreafico,
\emph{Axiomatic theory for transversality and Bertini type theorems},
Arch.~Math. (Basel) \textbf{70} (1998),  no.~5, 407--424.


\bibitem{St77}
R.~Stanley,
Cohen--Macaulay complexes,
in \emph{Higher Combinatorics} (M.~Aigner ed.),
Reidel, Dordrecht and Boston, 1977, 51--62.

\bibitem{St96}
R.~Stanley,
\emph{Combinatorics and Commutative Algebra},
Birkh\"auser, 1996.

\bibitem{StVo}
J.~St\"uckrad and W.~Vogel,
\emph{Buchsbaum Rings and Applications},
Springer-Verlag, Berlin, 1986.

\bibitem{Trung}
N.V.~Trung,
\emph{Toward a theory of generalized Cohen--Macaulay modules},
Nagoya Math.~J. \textbf{102} (1986), 1--49.

\bibitem{yanagawa}
K.~Yanagawa,
\emph{Alexander duality for Stanley-Reisner rings and squarefree
$\NN^n$-graded modules}, J.~Algebra \textbf{225} (2000), no.~2, 630--645.

\bibitem{Yuk}
Y.~Takayama,
\emph{Combinatorial characterizations of generalized Cohen--Macaulay ideals},
Bull.~Math.~Soc.~Sci.~Math.~Roumanie \textbf{48}(96) (2005), 327--344.

\end{thebibliography}
\end{document}